\documentclass[13pt,a4paper]{amsart}
\usepackage{mathrsfs}
\usepackage{CJK, bbm}
\usepackage{amsfonts}
\usepackage{xy}
\usepackage{amssymb}
\usepackage{color, latexsym}
\usepackage{graphicx}
\usepackage[colorlinks=true,
           citecolor=blue,
           linkcolor =red]{hyperref}

  \allowdisplaybreaks

\fontsize{9pt}{9pt}\selectfont

\numberwithin{equation}{section}

\swapnumbers
\newtheorem{theorem}[equation]{Theorem}
\newtheorem{lemma}[equation]{Lemma}

\newtheorem{prop-def}[equation]{Proposition-Definition}

\theoremstyle{definition}
\newtheorem{definition}[equation]{Definition}

\theoremstyle{remark}
\newtheorem{remark}[equation]{Remark}

\def\Hom{\mathrm{Hom}}
\def\Mod{\mathrm{Mod}}
\def\End{\mathrm{End}}
\def\Gr{\mathrm{Gr\,}}

\newenvironment{point}[2]%
  {\ifx*#2\let\pointlabel\relax\else\def\pointlabel{#2}\fi
   \refstepcounter{equation}\trivlist
   \item[\hskip\labelsep\theequation.
         \ifx\pointlabel\relax\else\space\pointlabel\space\fi]
   \ignorespaces #1
  }{\relax}

\begin{document}
\setlength{\itemsep}{-0.25cm}
\fontsize{12}{\baselineskip}\selectfont
\setlength{\parskip}{0.34\baselineskip}
\vspace*{-12mm}
\title[Clifford Superalgebras]{\fontsize{13}{\baselineskip}\selectfont Graded Morita equivalence of Clifford superalgebras}
\author{Deke  Zhao}
\address{School of Applied Mathematics, Beijing Normal University at Zhuhai,
Zhuhai 519087, China} \email{deke@amss.ac.cn}
\thanks{The author is supported in part by NSFC (No.~11001253 and No.~11101037).}


\begin{abstract}
This note uses a variation of graded Morita theory for finite dimensional superalgebras to determine explicitly
 the graded basic superalgebras for all real and complex Clifford superalgebras. As an application,
 the Grothendieck groups of the category of left $\mathbb{Z}_2$-graded modules over all real and
 complex Clifford superalgebras are described explicitly.
\end{abstract}
\maketitle
\vspace*{-10mm}
\section{Introduction}
Clifford (super)algebras or special Grassmann (super)algebras play an important role in many branches of
mathematics such as Clifford analysis, algebras, mathematics physics, geometry and topology etc., see for
example \cite{abs,V}.  For the finite dimensional superalgebras, there is a natural graded Morita theory,
which states that a finite dimensional superalgebra is graded Morita  equivalent to a
 graded basic superalgebra, see \cite[Theorem 2.3]{han} for a more general context.

     Note that some important algebraic objects such as
   the Hochschild and cyclic (co)homology of superalgebras are graded Morita invariants \cite{zhao} and
   that the fundamental relationship between Clifford (super)algebras and Bott periodicity is established by
 Atiyah, Bott and Shapiro \cite{abs}, see also \cite{Lawson} for a exposition. It is a natural and interesting problem
   to determine explicitly the graded basic superalgebras for Clifford superalgebras. In this note, using a
    variation of the graded Morita equivalent theory for finite dimensional
superalgebras (Proposition-Definition~\ref{psbasicdef}), we determine explicitly the graded basic superalgebras for
all real and complex Clifford superalgebras (Theorems~\ref{tclifford}
 and \ref{thm-complex}). As an application, we determine explicitly the Grothendieck groups of $\mathbb{Z}_2$-graded modules over all real and complex Clifford superalgebras (Theorem~\ref{thm: Grothendieck}), which is useful to our
  understanding of the Atiyah-Bott-Shapiro isomorphisms.

The lay-out of this note as follows. We begin in Section~\ref{sec: Preliminaries} with preliminaries on superalgebras
 and fix our notations. In section~\ref{sec: Morita} we review briefly the graded Morita theory for superalgebras
 and give some useful lemmas. The graded basic superalgebras for all real and complex Clifford
  superalgebras are determine explicitly in section~\ref{sec: Main}. Finally, as an application, we determine
   explicitly the Grothendieck groups of $\mathbb{Z}_2$-graded modules
   over all real and complex Clifford superalgebras in section~\ref{application}.

\section{Preliminaries on superalgebras}\label{sec: Preliminaries}
 In this section, we recall some facts on superalgebras and fix the notations, our references are
  \cite{bass}, \cite[Chapter 3]{manin} and \cite[\S 1]{kassel}.

 \begin{point}{}*
Let $K$ be a field and $\mathbb{Z}_2:=\mathbb{Z}/2\mathbb{Z}=\{0,
1\}$. By a \textit{superspace} we mean a $\mathbb{Z}_2$-graded
$K$-vector space $V$, namely a $K$-vector space with a decomposition
into two subspaces $V = V_{0}\oplus V_{1}$. A nonzero element $v$ of
$V_i$ will be called \textit{homogeneous} and we denote its degree by
$|v|=i\in \mathbb{Z}_2$. We will view  $K$ as a superspace concentrated in degree 0.

Given   superspaces $V$ and $W$, we view the direct sum $V\oplus W$
and the tensor product $V\otimes_K W$  as superspaces with $(V\oplus
W)_i = V_i\oplus W_i$, and $(V\otimes_K W)_i= V_0\otimes_K W_i\oplus
V_1\otimes_K W_{1-i}$  for $i\in\mathbb{Z}_2$. With this grading,
$V\otimes_K W$ is called the \textit{skew tensor product} of $V$ and
$W$ and is denoted by $V\widehat{\otimes}_KW$. Also, we make the
vector space $\Hom_K(V, W)$ of all $K$-linear maps from $V$ to $W$
into a superspace by setting that $\Hom_K(V, W)_i$ consists of all
the $K$-linear maps $f: V \rightarrow W$ with $f(V_j)\subseteq
W_{i+j}$ for $i, j\in \mathbb{Z}_2$. Elements of
$\Hom_K(V, W)_0$ (resp.~$\Hom_K(V, W)_1$) will be referred to as
\textit{even (resp.~odd) linear maps}.\end{point}

\begin{point}{}* Recall that a \textit{superalgebra} $A$ is both a superspace
and  an associative  $K$-algebra  with identity such that
$A_iA_j\subseteq A_{i+j}$ for $i, j \in \mathbb{Z}_2$. By a
superalgebra \textit{homomorphism} we mean an even linear map which is an
algebra homomorphism in the usual sense and write $A\cong B$ provided that the superalgebras
$A$ and $B$ are isomorphic.

Given  two  superalgebras $A$ and $B$, the \textit{skew tensor
product} $A\widehat{\otimes}_{K} B$  is again a superalgebra with
the inducing grading and multiplication given by
\begin{align*}(a_1 {\otimes} b_1)(a_2 {\otimes} b_2) =
(-1)^{|b_1||a_2|}a_1a_2{\otimes} b_1b_2, \text{ for } a_i\in A \text{ and } b_i\in B.\end{align*}
Note this and other such expressions only make sense for
homogeneous elements.
Observe that the skew tensor product $A\widehat{\otimes}_{K} A\widehat{\otimes}_{K}\cdots\widehat{\otimes}_{K}A$ ($n$ factors) is well-defined and that the \textit{supertwist
map} $T_{A,B}: A\widehat{\otimes}_K B\rightarrow
B\widehat{\otimes}_K A$, $a\otimes b \mapsto (-1)^{|a||b|}b\otimes
a$, for  $a\in A, b \in B$, is an isomorphism of superalgebras. \end{point}

\begin{point}{}*The main principle for superalgebras is the following  \textit{rule
of signs} \cite[\S III.4]{good} or \cite[\S3.1.1]{manin}: if in some formula of usual algebra there are
monomials with interchanged terms, then in the corresponding formula
in superalgebra every interchange neighboring terms, say $a$ and
$b$, is accompanied by the multiplication of the monomial by the
factor $(-)^{|a||b|}$; or equivalently, each letter $a_i$ appearing
as an argument on the left-hand side (LHS) of the defining equation
has a degree associated with it, and a factor of (-1) is introduced
on the right-hand side (RHS) each time a pair of letters on the LHS,
both of odd degree, appearing in reverse order on the RHS. In order
for this rule make sense it is essential that every letter on the
LHS should appear exactly once on the RHS.
\end{point}

\begin{point}{}*
 Let $A$ be a superalgebra. By  a
graded $A$-module $M$ we mean a $\mathbb{Z}_2$-graded left
$A$-module, that is, $M$ is both a superspace $M = M_0\oplus  M_1$
and a left $A$-module such that $A_iM_j \subseteq M_{i+j}$ for
$i, j \in \mathbb{Z}_2$. We denote by $\widehat{A}$ the superalgebra with the same underlying
superspace as $A$ but new multiplication $\hat{a}\hat{b}=(-1)^{|a||b|}\widehat{ab}$.
If $M$ is a graded $A$-module, let $\widehat{M}$ denote the
$\widehat{A}$-module with $M$ as the underlying superspace and
operators defined as $\hat{a}\hat{m}\!=\!(-1)^{|a||m|}\widehat{am}$ for
$a\in\!A,m\in\!M$.

 Unless  otherwise explicitly stated, all our
modules will be $\mathbb{Z}_2$-graded left modules. We denote by
$\mathrm{Mod} A$ the category of all $A$-modules with morphisms
\begin{align*}\Hom_{\Mod A}(M, N):=\Hom_{\Mod A}(M,
N)_0+\Hom_{\Mod A}(M,N)_1,\end{align*}
where $\Hom_{\Mod A}(M, N)_i$, $i\in\mathbb{Z}_2$, is consisting of
all $K$-linear maps $f$ from $M$ to $N$
 such that $f(M_j)\subseteq N_{i+j}$ and $f(am)=(-1)^{i|a|}af(m)$,
 for all $a\in A$, $m\in M$ and  $j\in\mathbb{Z}_2$. The elements of $\Hom_{\Mod A}(M,
N)_0$ (resp.~$\Hom_{\Mod A}(M, N)_1$)  are called \textit{even (resp.~odd) homomorphisms} form $M$ to $N$. Let $\mathrm{Gr} A$ be the category of all $A$-modules with even
homomorphisms. \end{point}

\section{Graded Morita equivalent theory}\label{sec: Morita}
\begin{definition}\label{shifts}
Let $A$ be a superalgebra. The \textit{ parity change} (reps.~\textit{suspension})  functor $\pi$ (resp.~$\sigma$) from $\mathrm{Gr} A$ to itself is defined as following:  for
$M=M_0\oplus M_1$ in $\mathrm{Gr} A$, we define the graded
$A$-module $\pi(M)$ (resp.~$\sigma(M)$) by the conditions:
\begin{enumerate}
\item $\pi(M)$ and $\sigma(M)$ are superspaces with $\pi(M)_i=\sigma(M)_i=M_{i+1}$ for $i\in \mathbb{Z}_2$;

\item the module structure on $\pi(M)$ (resp.~$\sigma(M)$) is defined by $a\pi(m):=(-1)^{|a|}\pi(am)$ (resp. $a\sigma(m):=\sigma(am)$)
for homogeneous elements $a\in A$ and $m\in M$.
\end{enumerate}
\vspace{-2.3mm}
For a homomorphism $f$, we define $\pi(f)$ (resp.~$\sigma(f)$) is the
same underlying linear map as $f$.
\end{definition}

\begin{remark} It follows by definition that $\pi^2=\mathrm{Id}_{\mathrm{Gr} A}$ and $\sigma^2=
\mathrm{Id}_{\mathrm{Gr}A}$, which means that the parity change functor $\pi$ and the suspension functor
$\sigma$ are shifts of the category $\mathrm{Gr}A$ in the sense of \cite[Definition 3.2]{zhang} and \cite[\S2]{s}.  \end{remark}

From now on we assume that $S$ is either the parity change functor $\pi$ or the suspension functor
$\sigma$, unless otherwise explicitly stated. For  an  $A$-module $M$, following \cite[\S2]{s}, we
define the $S$-{\it twisted endomorphism superalgebra} of $M$  to be
the superalgebra
\begin{align*}\mathrm{End}_A^{S}(M):=\Hom_{\Gr A}(M,
M)\oplus\Hom_{\Gr A}(S(M), M),\end{align*}
and define the $S$-{\it twisted $\Hom$ functor} to be the
functor\begin{align*}\Hom^{S}_A(M, -): \Gr A\rightarrow \Gr
\End_A^{S}(M),\quad  N\mapsto \Hom_{\Gr A}(M,
N)\oplus\Hom_{\Gr A}(S(M), N).\end{align*}
Denote by $\Mod^{S} A$ the category of all graded $A$-modules
with homomorphisms $\Hom_A^{S}(M, N)$.

\begin{lemma}[cf.~\cite{bass}, p.\!~123] Assume that $A$ is a finite dimensional superalgebras and that
$M$ and $N$ are $A$-modules. Then $\mathrm{Hom}_{\Mod^{\pi} A}(M, N)\!=\!\Hom_{\Mod^{\sigma}
\widehat{A}}(\widehat{M}, \widehat{N})$ and $\Mod^{\pi} A=\Mod A=\Mod^{\sigma}\widehat{A}$.\label{lbass}\end{lemma}

\begin{proof}First note that for $A$-modules $M$ and $N$, by definition,
\begin{align*}&\mathrm{Hom}_{\Mod^{\pi} A}(M, N)=\Hom_{A}^{\pi}(M,N)=\Hom_{\Gr A}(M,N)\oplus\Hom_{\Gr A}(\pi(M),N) \text{ and}\\
&\Hom_{\Mod^{\sigma}
\widehat{A}}(\widehat{M}, \widehat{N})=\Hom_{\widehat{A}}^{\sigma}(\widehat{M},\widehat{N})=\Hom_{\Gr \widehat{A}}(\widehat{M},\widehat{N})\oplus\Hom_{\Gr \widehat{A}}(\sigma(\widehat{M}),\widehat{N}). \end{align*}
Secondly by definition, we have
\begin{align*}\Hom_{\Gr \widehat{A}}(\widehat{M},\widehat{N})&=\{f: \widehat{M}\rightarrow\widehat{N}\mid f(\widehat{M}_i)\subseteq\widehat{N}_i \text{ and } f(\hat{a}\hat{m})=\hat{a}f(\hat{m}), \forall  i\in\mathbb{Z}_2, a\in A, m\in M\!\}\\
&=\{f: M\rightarrow N\mid f(M_i)\subseteq N_i \text{ and }f(am)=af(m), \forall  i\in\mathbb{Z}_2, a\in A, m\in M\!\}\\
&=\Hom_{\Gr A}(M,N);\\
\Hom_{\Gr \widehat{A}}(\sigma(\widehat{M}),\widehat{N})&=\!\!\{\!f\!: \sigma(\widehat{M})\!\rightarrow\!\widehat{N}\!\mid\! f(\sigma(\widehat{M})_i)\!\subseteq\!\widehat{N}_i\text{ and } f(\hat{a}\sigma(\hat{m}))\!\!=\!\!\hat{a}f(\sigma(\hat{m})), \forall  i\!\in\!\mathbb{Z}_2,a\!\in\! A, m\!\in\! M\!\}\\
&=\!\!\{\!f\!: M\!\rightarrow\! N\!\mid\! f(M_i)\!\subseteq\! N_{1+i}\text{ and }f(am)\!=\!(-1)^{|a|}af(m), \forall  i\!\in\!\mathbb{Z}_2,a\!\in\! A, m\!\in\! M\}\\
&=\Hom_{\Gr A}(\pi(M),N).
\end{align*}
Therefore  $\mathrm{Hom}_{\Mod^{\pi} A}(M, N)\!=\!\Hom_{\Mod^{\sigma}
\widehat{A}}(\widehat{M}, \widehat{N})$.

Finally note that the objects of the categories $\Mod^{\pi} A$, $\Mod A$ and $\Mod^{\sigma}\widehat{A}$ are same.
By the first part of the lemma, we only need to show that $\Hom_{\Mod^{\pi} A}(M,N)=\Hom_{\Mod A}(M,N)$ for all $A$-modules $M$ and $N$; furthermore, it suffices to show that $\Hom_{\Gr A}(\pi(M),N)=\Hom_{\Mod A}(M,N)_{1}$ for all $A$-modules $M$ and $N$. Indeed for all $A$-modules $M$ and $N$, by definition,
\begin{align*}\Hom_{\Gr A}(\!\pi(M),N\!)&=\!\{\!f\!:\pi(\!M\!)\!\rightarrow\! N\!\mid\!f(\!\pi(M)_i\!)\!\subseteq\!N_i \text{ and } f(a\pi(m))\!=\!af(\pi(m)), \forall i\!\in\! \mathbb{Z}_2, a\!\in\! A, m\!\in\! M\!\}\\
&=\!\{\!f\!:M\!\rightarrow\! N\!\mid\!f(\!M_{i+1}\!)\!\subseteq\!N_{i} \text{ and } f(\!\pi(am)\!)\!=\!(\!-1\!)^{|a|}af(\!\pi(m)\!), \forall i\!\in\! \mathbb{Z}_2, a\!\in\! A, m\!\in\! M\!\}\\
&=\!\{\!f\!:M\!\rightarrow\! N\!\mid\!f(M_i)\!\subseteq\!N_{1+i} \text{ and } f(am)\!=\!(-1)^{|a|}af(m), \forall i\!\in\! \mathbb{Z}_2, a\!\in\! A, m\!\in\! M\!\}\\
&=\Hom_{\Mod A}(M,N)_{1},\end{align*}
where the second and third equalities follow by Definition~\ref{shifts} that $\pi(M)_i=M_{i+1}$, $a\pi(m)=(-1)^{|a|}\pi(am)$, and that $f(\pi(am)=f(am)$ for homogeneous elements $a\in A$, $m\in M$ and homogeneous homomorphisms $f$.
As a consequence, we complete the proof.
\end{proof}

 We say that a non-zero $A$ module $M$ is \textit{gr-indecomposable} if it is not the direct sum of two
non-zero   modules that  $M$ is
\textit{$S$-indecomposable} if it is not the direct sum of two
non-zero $A$-modules in the category $\mathrm{Mod}^{S} A$.  A
superalgebra is called to be \textit{gr-divisional} (resp.~\textit{gr-local}) if
every nonzero homogeneous element of it is invertible (resp.~either
invertible or nilpotent).

\begin{lemma}[\cite{han}, \S2.2] Let $A$ be a finite dimensional
superalgebra over $K$ and $M\in\Gr A$. Then
\begin{enumerate}\item $M$ is $S$-indecomposable if and only if $M$ is gr-indecomposable.

\item $M$ is gr-indecomposable if and only if
 $\mathrm{End}^S_{A}(M)$ is gr-local.

\item If $M$ is  gr-simple  then
$\mathrm{End}^S_{A}(M)$ is either a gr-divisional superalgebra
concentrated in degree zero, or a gr-divisional superalgebra containing an
odd involuting element.
\end{enumerate}
\end{lemma}

\begin{definition} Let $A$ and $B$ be two finite dimensional
superalgebras over $K$. We say that $A$ and $B$ are $S$-\textit{graded equivalent}, denote
by $A\thickapprox_{S}B$, if the categories $\Gr A$  and $\Gr B$ are equivalent and $B\!\cong\!\End_A^{S}(P)$
for some finitely generated projective $A$-module $P$. \end{definition}

The relationship between $\pi$-graded equivalence and $\sigma$-graded equivalence is the following:
\begin{lemma}\label{rsp}Assume that finite dimensional superalgebras $A$ and $B$ are $\sigma$-graded equivalent.
Then $A$ and $\widehat{B}$ are $\pi$-graded equivalent.
\end{lemma}

\begin{proof}Assume that the categories $\Gr A$ and $\Gr B$ are equivalent and that there is a finitely
 generated projective $A$-module $P$ such that $B\cong \End_A^{\sigma}(P)$. Observe that
 superalgebras $B$ and $\widehat{\widehat{B}}$ are isomorphic and that $\Gr B$ and
 $\Gr \widehat{B}$ are equivalent.  Applying Lemma~\ref{lbass},
  $\End_{\widehat{A}}^{\pi}(\widehat{P})=\End_{\widehat{\widehat{A}}}^{\sigma}(\widehat{\widehat{P}})=
 \End_A^{\sigma}(P)$. The proof is completed.
\end{proof}

Let $A$ be a finite dimensional superalgebra. Then $A$ has a complete set of primitive orthogonal
idempotents and denoted it by  $\{f_1, \cdots, f_n\}$. Let $P_i=Af_i$ be the projective $A$-module corresponding to the primitive orthogonal idempotent $f_i$ of $A$. Then $\widehat{P}_i=\widehat{A}f_i$ and Lemma~\ref{lbass} and \cite[Corollary 3.10]{Da} imply that $\mathrm{Hom}_{\Mod^\pi A}(P_i,
P_j)=\Hom_{\Mod^\sigma A} (\widehat{P}_i,
\widehat{P}_j)=f_i\widehat{A}f_j $  and $\Hom_{\Mod^\sigma A} (P_i,
P_j)=f_iAf_j$.

We say
that two idempotents $f$ and $g$ of $A$ are $S$-\textit{equivalent}
if the $A$-modules $Af$ and $Ag$ are isomorphic in the category $\Mod^S A$. The following is a criterion of $S$-equivalent of idempotents.

\begin{lemma}\label{lequivalent}The idempotents  $f$ and $g$ are
$S$-equivalent if and only if
 there exist homogeneous elements $x\in fAg$ and $y\in gAf$ satisfying $xgy=f$ and
 $yfx=g$.
\end{lemma}
\begin{proof}Note that $\Hom_{\Mod^{S}A}(Af, Ag)$ is either $fAg$ or $f\widehat{A}g$ and that $\Hom_{\Mod^{S}A}
(Ag, Af)$ is either $gAf$ or $g\widehat{A}f$. Assume that $Af$ and $Ag$ are isomorphic in the category $\Mod^{S}A$.
Then there exists invertible
homogenous homomorphisms induced by homogeneous elements $x\in fAg$ and
$y\in gAf$ satisfying $xgy=f$ and  $yfx=g$.
Conversely, suppose that there exist homogeneous elements $x\in fAg$ and $y\in gAf$ satisfying $xgy=f$ and
 $yfx=g$. Then $x$ and $y$ induce natural homomorphisms $Af\rightarrow Ag$, $f\mapsto xg$ and $Ag\rightarrow Af$, $g\mapsto yf$ respectively, which give the desired isomorphisms.
\end{proof}

The following definition is our investigation for all Clifford superalgebras in this note.
\begin{prop-def}[cf.~Proposition 2.4, \cite{han}] Let $A$ be a finite dimensional superalgebra and $J(A)$
the graded Jacobson radical of $A$. We say that the superalgebra $A$ is \textbf{$\mathbf{S}$-graded basic}
if it satisfies the following equivalent conditions:

\begin{enumerate} \item $f_i$  are pairwise non-$S$-equivalent for $1\le
i\le n$.

\item  Any decomposition of $A$ into indecomposable
projective $A$-modules $A = \oplus^n_{i=1}P_i$ satisfies $P_i$ and $P_j$ are not isomorphic in $\Mod^{S}A$ for all $1\le i \neq j\le n$.

 \item $A/J(A)$ is a direct sum of gr-divisional superalgebras.
 \end{enumerate}
\label{psbasicdef}\end{prop-def}

\begin{remark}By \cite[Theorem 2.3]{han}, a finite dimensional superalgebra $A$
 is $S$-graded equivalent to an $S$-graded basic superalgebra. More precisely, let
 $\{e_i\mid 1\leq i\leq n\}$ be a complete set of non-$S$-equivalent orthogonal
 primitive idempotents of $A$. It follows that $A$ is $S$-graded equivalent to
 $\End^S_A(\oplus_{i=1}^nAe_i)$, which is $S$-graded basic.
  \end{remark}

 The following Lemma is the key to the graded Morita classification for Clifford superalgebras.
\begin{lemma}\label{pskewtensor}
Suppose that the superalgebras $A$ and $A'$ are $S$-graded
equivalent to  $B$ and $B'$ respectively. Then $A\widehat{\otimes}_K
A'$ is $S$-graded equivalent to $B\widehat{\otimes}_K B'$.
\end{lemma}

\begin{proof}Without loss of generality, we may assume that both $B$ and $B'$ are $S$-graded basic and that
$\{f_1, \dots, f_r\}$ and $\{f'_1, \dots,
f'_{r'}\}$ are the complete sets of orthogonal primitive idempotents of $B$ and $B'$ respectively.
Now let $\{e_1, \dots, e_n\}$ and $\{e'_1, \dots,
e'_{n'}\}$ be the complete sets of orthogonal primitive idempotents of
$A$ and $A'$ respectively. Then  $\{f_1, \dots, f_r\}\subseteq\{e_1, \dots, e_n\}$ and $\{f'_1, \dots,
f'_{r'}\}\subseteq\{e'_1, \dots,e'_{n'}\}$. Observe that
$\{e_i\otimes e'_j| 1\le i\le n, 1\le j\le n' \}$ and
$\{f_i\otimes f'_j|1\le i\le r, 1\le j\le r'\}$ is a set of
orthogonal idempotents of $A\widehat{\otimes}_K A'$ such that $\sum_{i,j}e_i\otimes e'_j=1_{A}\otimes 1_{A'}$, and
$\{f_i\otimes f'_j|1\le i\le r, 1\le j\le r'\}$ is a set of
orthogonal idempotents of $B\widehat{\otimes}_K B'$ such that $\sum_{i,j}f_i\otimes f'_j=1_{B}\otimes 1_{B'}$.
Note that $\{f_i\otimes f'_j|1\le i\le r, 1\le j\le r'\}\subseteq \{e_i\otimes e'_j| 1\le i\le n, 1\le j\le n' \}$, which implies that the complete set of non-$S$-equivalent
orthogonal primitive idempotents of $A\widehat{\otimes}_K A'$ can be obtained form
$\{f_i\otimes f'_j|1\le i\le r, 1\le j\le r'\}$ by decomposing these orthogonal idempotents. As a consequence,
$A\widehat{\otimes}_K A'$ is $S$-graded  equivalent to
$B\widehat{\otimes}_K B'$.
\end{proof}

We close this section with some remarks.
 \begin{remark}\label{Rem:graded-basic}\begin{enumerate}\item The $\sigma$-graded equivalence
is exactly the so-called \textit{graded Morita equivalence} defined in \cite{zhang,s} for group-graded
 algebras by ignoring the rule of signs in the case of superalgebras.

 \item A finite dimensional superalgebra is $\sigma$-graded equivalent but not $\pi$-graded equivalent to itself.

 \item The rule of signs implies that for finite dimensional superalgebra
 the $\pi$-graded equivalence is deserving of a better understanding than $\sigma$-graded equivalence, that is, for finite dimensional superalgebra
 $A$ the category $\Mod A$ should be study exhaustively.

\item\label{item-pi-sigma} Lemmas~\ref{rsp} and \ref{lequivalent} show that the $S$-graded basic classification of finite dimensional superalgebras can be determined completely by its graded Morita equivalent classification, that is, by its  $\sigma$-graded equivalent classification.
     \end{enumerate}\end{remark}
\section{Graded Morita equivalent classification of Clifford superalgebras}\label{sec: Main}
 From now on we will write $\mathbb{R}$, $\mathbb{C}$
and $\mathbb{H}$ respectively for the real, complex and quaternion
number-fields and view them as superalgebras over $\mathbb{R}$
concentrated on degree zero respectively and we will say a superalgebra to be graded basic if it is either
$\sigma$-graded basic or $\pi$-graded basic (cf.~Remark~\ref{Rem:graded-basic}(iv)). Note that $\mathbb{R}$, $\mathbb{C}$
and $\mathbb{H}$ are gr-divisional superalgebras; in particular, they are graded basic superalgebras according to Proposition-Definition~\ref{psbasicdef}.

\begin{point}{}*\label{pchevalley} Let $p$, $q$ and $r$ be positive integers.
Following Porteous \cite{porteous95}, we
 denote by $\mathbb{R}^{p,
q, r}$ the real quadratic space $\mathbb{R}^{p+q+r}$ with the
quadratic form $\rho(x)=x_1^2+\cdots+ x_{p}^2-x_{p+1}^2-\cdots-x_{p+q}^2$.
The Clifford superalgebra  $\mathbb{R}_{p, q, r}$  on
$\mathbb{R}^{p, q, r}$ is the real unitary superalgebra generated by
odd generators $e_1, \dots, e_{p+q+r}$ subject to the
following relations:\begin{align*}& e_ie_j+e_je_i=0&& \text{for } 1\leq i\neq j\leq p+q+r;\\
& e_i^2=-e_{j+p}^2=1&&\text{for }1\leq i\leq p \text{ and }1\leq j\leq q;\\
&e_{p+q+k}^2=0&&\text{for } 1\le k\leq r.\end{align*}
Note that the Clifford superalgebra
$\mathbb{R}_{0, 0, r}$ is the real Grassmann superalgebra
$\bigwedge_{\mathbb{R}}(r)$ generated by odd generators $e_{p+q+1}, \dots, e_{p+q+r}$. Observe that the orthogonal primitive idempotent of $\bigwedge_{\mathbb{R}}(r)$ is the unit 1 since $\bigwedge_{\mathbb{R}}(r)/J(\bigwedge_{\mathbb{R}}(r))=\mathbb{R}$, which implies $\bigwedge_{\mathbb{R}}(r)$
are  graded basic  for all positive integer $r\ge 1$ according to Proposition-Definition~\ref{psbasicdef}.

For now on we write $\mathbb{R}_{p, q}$ for the Clifford superalgebra
$\mathbb{R}_{p, q, 0}$. It is well-known that there are superalgebras isomorphism $\mathbb{R}_{p, q, r}
\cong \mathbb{R}_{p, q}\widehat{\otimes}_\mathbb{R}\bigwedge_{\mathbb{R}}(r)$ and $\mathbb{R}_{p,
q}\widehat{\otimes}_{\mathbb{R}} \mathbb{R}_{p', q'}\cong
\mathbb{R}_{p+p', q+q'}$ for all positive integers $p$, $q$, $r$, $p'$ and $q'$, see \cite{chev} or \cite[Proposition~1.6]{abs}.\end{point}

\begin{point}{}* Let $V=\mathbb{R}^{p+q+r}$ and $\alpha: V\rightarrow V$, $v\mapsto -v$. Then $\alpha$ induces a
 automorphism of the (ungraded) Clifford algebra $\mathbf{C}\ell(V,\rho)=T(V)/\langle v\otimes v-\rho(v)|v\in V\rangle$
  where $T(V)$ is the tensor algebra of $V$. Since $\alpha^2=\mathrm{Id}$, there is a decomposition
\begin{align*}\mathbf{C}\ell(V,\rho)=\mathbf{C}\ell^0(V,\rho)\oplus \mathbf{C}\ell^1(V,\rho)\end{align*}
where $\mathbf{C}\ell^i(V,\rho)=\{\phi\in\mathbf{C}\ell(V,\rho)\mid\alpha(\phi)=(-1)^i\phi\}$ are the eigenspaces
of $\alpha$. Clearly, since $\alpha(\phi_1\phi_2)=\alpha(\phi_1)\alpha(\phi_2)$, we have that
\begin{align*}\mathbf{C}\ell^i(V,\rho)\mathbf{C}\ell^j(V,\rho)\subseteq \mathbf{C}\ell^{i+j}(V,\rho)\quad
\text{ for all }i,j\in \mathbb{Z}_2,\end{align*}
that is, $\mathbf{C}\ell(V,\rho)$ is a superalgebra.  It is an observation of Atiyah, Bott and Shapiro \cite{abs}
 that this $\mathbb{Z}_2$-grading plays an important role in the analysis and application of Clifford algebras.
\end{point}

\begin{remark}The two superalgebras $\mathbb{R}_{p, q, r}$ and $\mathbf{C}\ell(V,\rho)$ are naturally isomorphic.
\end{remark}

Denote by $\mathbb{D}_{+}$ (resp.~$\mathbb{D}_{-}$) the real superalgebra generated by odd element $e_{+}$
(resp.~$e_{-}$) subject to the relation $e_{+}^2=1$ (resp.~$e^2_{-}=-1$). For a positive integer $n$, we denote by
$\mathbb{D}_{+}^{n}$ (resp.~$\mathbb{D}_{-}^{n}$) the superalgebra
$\mathbb{D}_{+}\widehat{\otimes}_{\mathbb{R}}\cdots\widehat{\otimes}_{\mathbb{R}}
\mathbb{D}_{+}$  (resp.~$\mathbb{D}_{-}\widehat{\otimes}_{\mathbb{R}}\cdots\widehat{\otimes}_{\mathbb{R}}
\mathbb{D}_{-}$) ($n$ factors) and write  $e_{+}^{i_1\dots i_n}$ (resp.~$e_{-}^{i_1\dots i_n}$) for the homogeneous element
$e_{+}^{i_1}\otimes\dots\otimes
e_{+}^{i_n}\in\mathbb{D}_{+}^{n}$ (resp.~$e_{-}^{i_1}\!\otimes\!\cdots\!\otimes\!
e_{-}^{i_n}\!\in\!\mathbb{D}_{-}^{n}$) where $i_j=0, 1$ for all $1\leq j\leq n$. Note that $\mathbb{D}_{+}^n$
 and $\mathbb{D}_{-}^n$ are exactly $\mathbb{R}_{n,0}$ and $\mathbb{R}_{0,n}$ respectively.

\begin{lemma}\label{dd}If $n=1, 2,3$ then $\mathbb{D}^n_{+}$ and $\mathbb{D}^n_{-}$ are
gr-divisional superalgebras.\end{lemma}
\begin{proof}First, it is trivial that $\mathbb{D}_{\pm}$  are gr-divisional superalgebras. Noticing that the even (resp.~odd) elements of $\mathbb{D}_{+}^2$ are of the form $a\otimes 1+be_{+}
\otimes e_{+}$ (resp.~$a\otimes e_{+}+be_{+}\otimes 1$), where $a$, $b\in \mathbb{R}$, and that for
all $a, b\in \mathbb{R}$, \begin{align*}(a\otimes 1+be_{+}\otimes e_{+})
 (a\otimes 1-be_{+}\otimes e_{+})=(a^2+b^2)\otimes 1=
 (a\otimes e_{+}+be_{+}\otimes 1)(a\otimes e_{+}+be_{+}\otimes
 1).\end{align*}
Thus all non-zero homogeneous elements of $\mathbb{D}_{+}^2$ are invertible, that is,
$\mathbb{D}_{+}^2$ is gr-divisional.

Set $\mathbf{1}:=1\widehat{\otimes}1\widehat{\otimes}1$,
$\mathbf{i}_{+}:=e_{+}^{011}$,
$\mathbf{j}_{+}:=e_{+}^{101}$,
$\mathbf{k}_{+}:=e_{+}^{110}$ and
$\theta_{+}:=e_{+}^{111}$. Then it follows directly that $\theta_{+} \mathbf{i}_{+}+\mathbf{i}_{+}\theta_{+}=0$, $\theta_{+} \mathbf{j}_{+}+\mathbf{j}_{+}\theta_{+}=0$ and $\theta_{+} \mathbf{k}_{+}+\mathbf{k}_{+}\theta_{+}=0$. Furthermore, we have $\mathbb{R}\langle\mathbf{1}, \mathbf{i}_{+},
\mathbf{j}_{+}, \mathbf{k}_{+}\rangle\cong \mathbb{H}$ and
$\mathbb{D}_{+}^3\cong\mathbb{H}\oplus\mathbb{H}\theta_{+}$. Note that $\theta_{+}$ is a odd element of $\mathbb{D}^3_{+}$ satisfying $\theta^2_{+}=-1$ and that $\mathbb{H}$ is a gr-divisional superalgebra concentrated on degree zero.  So $\mathbb{D}^3_{+}$ is  gr-divisional.

 Observe that the even (resp.~odd) elements of $\mathbb{D}_{-}^2$ are of the form $a\otimes 1+be_{-}\otimes e_{-}$ (resp.~$a\otimes e_{-}+be_{-}\otimes 1$), where $a$, $b\in \mathbb{R}$, and that for all $a, b\in \mathbb{R}$,
 \begin{align*}(a\otimes 1+be_{-}\otimes e_{-})
 (a\otimes 1-be_{-}\otimes e_{-})=-(a^2+b^2)\otimes 1=
 (a\otimes e_{-}+be_{-}\otimes 1)(a\otimes e_{-}+be_{-}\otimes
 1).\end{align*}
Thus all non-zero homogeneous elements of $\mathbb{D}_{-}^2$ are invertible, that is
$\mathbb{D}_{-}^2$ is gr-divisional.

Set $\mathbf{i}_{-}:=e_{-}^{011}$,
$\mathbf{j}_{-}:=e_{-}^{101}$,
$\mathbf{k}_{-}:=e_{-}^{110}$, and
$\theta_{-}:=e_{-}^{111}$. Then it follows directly that $\theta_{-} \mathbf{i}_{-}+\mathbf{i}_{-}\theta_{-}=0$, $\theta_{-} \mathbf{j}_{-}+\mathbf{j}_{-}\theta_{-}=0$, and $\theta_{-}\mathbf{k}_{-}+\mathbf{k}_{-}\theta_{-}=0$. Moreover, we have $\mathbb{R}\langle\mathbf{1}, \mathbf{i}_{-},
\mathbf{j}_{-}, \mathbf{k}_{-}\rangle\cong \mathbb{H}$ and
$\mathbb{D}_{-}^3\cong\mathbb{H}\oplus\mathbb{H}\theta_{-}$. Note that $\theta_{-}$ is a odd element of $\mathbb{D}^3_{-}$ satisfying $\theta^2_{-}=1$ and that $\mathbb{H}$ is a gr-divisional superalgebra concentrated on degree zero. Thus $\mathbb{D}^3_{-}$ is gr-divisional.
\end{proof}

\begin{lemma}\label{dc} $\mathbb{D}_{+}\widehat{\otimes}_{\mathbb{R}}\mathbb{D}_{-}$
is graded Morita equivalent to the superalgebra
$\mathbb{R}$.\end{lemma}
\begin{proof}Set $A=\mathbb{D}_{+}\widehat{\otimes}_{\mathbb{R}}\mathbb{D}_{-}$. By Proposition-Definition~\ref{psbasicdef},
it suffices to determine the non-$\sigma$-equivalent idempotents of $A$. Let
$f_{\pm}=\frac{1}{2}(1\otimes 1\pm e_{+}\otimes e_{-})$. Then $f_{\pm}^2=f_{\pm}$, $f_{+}+f_{-}=1\otimes 1$ and $f_{+}f_{-}=0=f_{-}f_{+}$, which implies that
$\{f_{+}, f_{-}\}$ is the complete set of orthogonal primitive
idempotents of $A$ since $\dim_{\mathbb{R}}A_{0}=2$. It is straightforward to show that
\begin{align*}f_{+}Af_{-}=\{r(e_{+}\otimes1-1\otimes e_{-})\mid r\in\mathbb{R}\}\quad\text{and}\quad
f_{-}Af_{+}=\{r(e_{+}\otimes1+1\otimes e_{-})\mid r\in\mathbb{R}\}.
 \end{align*}
Let $x=\frac{1}{2}(e_{+}\otimes1-1\otimes e_{-})$ and $y=\frac{1}{2}(e_{+}\otimes1+1\otimes e_{-})$. Then  $xf_{-}y=f_{+}$ and $yf_{+}x=f_{-}$.
Applying Lemma~\ref{lequivalent},  $f_+$ and $f_-$ are
$\sigma$-equivalent orthogonal primitive idempotents of $A$. As a consequence, Proposition-Definition~\ref{psbasicdef} implies that
$A$ is graded Morita equivalent to
$f_{+}Af_{+}=\mathbb{R}f_{+}\cong \mathbb{R}$.
\end{proof}

\begin{lemma}\label{dddd}$\mathbb{D}_{\pm}^4$ are graded Morita
equivalent to the superalgebra $\mathbb{H}$.
\end{lemma}

\begin{proof}By the proof of Lemma~\ref{dd}, the natural graded map $e_{-}\mapsto \theta_{+}$ gives an isomorphism between $\mathbb{D}_{-}{\widehat{\otimes}}_{\mathbb{R}}\mathbb{H}$ and $\mathbb{D}_{+}^3$.
 Therefore, using Lemmas~\ref{pskewtensor} and \ref{dc}, $\mathbb{D}_{+}^4=\mathbb{D}_{+}\widehat{\otimes}_{\mathbb{R}}
 \mathbb{D}_{+}^3\cong \mathbb{D}_{+}\widehat{\otimes}_{\mathbb{R}}\mathbb{D}_{-}\widehat{\otimes}_{\mathbb{R}}\mathbb{H}$
 is graded Morita equivalent to  the superalgebra $\mathbb{H}$. Similarly,
 noticing that $\mathbb{D}_{+}\widehat{\otimes}\mathbb{H}\cong \mathbb{D}_{-}^3$,
 we can show that $\mathbb{D}_{-}^4$ is graded Morita equivalent to the superalgebra $\mathbb{H}$.\end{proof}

\begin{remark}\begin{enumerate}\item The Lemmas~\ref{dc} and \ref{dddd} show that the skew tensor product
of $\sigma$-graded basic superalgebras may not be  $\sigma$-graded basic. In particular,
the skew tensor product of gr-divisional superalgebras may not be gr-divisional.
\item Lemmas~\ref{rsp}, \ref{dc} and \ref{dddd} imply that $\mathbb{D}_{+}\otimes\mathbb{D}_{-}$ and $\mathbb{D}_{\pm}^4$ are also $\pi$-graded equivalent to $\mathbb{R}$ and $\mathbb{H}$ respectively since $\widehat{\mathbb{D}_{\pm}}=\mathbb{D}_{\mp}$.
\end{enumerate}
\end{remark}
Now we can obtain the graded Morita classifications for all real Clifford superalgebras.

\begin{theorem}\label{tclifford} Assume that $p$, $q$ and $r$ are non-negative integers. Then \begin{enumerate}\item
$\mathbb{R}_{p, q, r}\,\approx_{\sigma}\left\{\aligned
&\bigwedge\!_{\mathbb{R}}(r), &&\text{if }p-q\equiv0 (\mathrm{mod\,} 8);\\
&\mathbb{H}\widehat{\otimes}_{\mathbb{R}}\bigwedge\!_{\mathbb{R}}(r),&&\text{if } p-q\equiv 4 (\mathrm{mod\,}8);\\
&\mathbb{D}_{+}^i\widehat{\otimes}_{\mathbb{R}}\bigwedge\!_{\mathbb{R}}(r), &&\text{if }p-q\equiv  4-i(\mathrm{mod\,}8)\text{ for } 1\leq i\leq 3;\\
&\mathbb{D}_-^i\widehat{\otimes}_{\mathbb{R}}\bigwedge\!_{\mathbb{R}}(r),&&\text{if }p-q\equiv 4+i (\mathrm{mod\,}8)
\text{ for } 1\le i \leq 3.
\endaligned\right.$

\item $\mathbb{R}_{p, q, r}\,\approx_{\pi}\left\{\aligned
&\bigwedge\!_{\mathbb{R}}(r), &&\text{if }p-q\equiv0 (\mathrm{mod\,} 8);\\
&\mathbb{H}\widehat{\otimes}_{\mathbb{R}}\bigwedge\!_{\mathbb{R}}(r),&&\text{if } p-q\equiv 4 (\mathrm{mod\,}8);\\
&\mathbb{D}_{+}^i\widehat{\otimes}_{\mathbb{R}}\bigwedge\!_{\mathbb{R}}(r), &&\text{if }p-q\equiv  4+i(\mathrm{mod\,}8)\text{ for } 1\leq i\leq 3;\\
&\mathbb{D}_-^i\widehat{\otimes}_{\mathbb{R}}\bigwedge\!_{\mathbb{R}}(r),&&\text{if }p-q\equiv 4-i (\mathrm{mod\,}8)
\text{ for } 1\le i \leq 3.
\endaligned\right.$
\end{enumerate}
$($Note that these superalgebras are graded basic superalgebras.$)$
\end{theorem}
\begin{proof}(i) Note that $\mathbb{D}_+\cong \mathbb{R}_{1, 0}$ and $\mathbb{D}_{-}\cong \mathbb{R}_{1, 0}$.
Using \cite[Proposition~1.6]{abs}, Lemmas~\ref{pskewtensor} and \ref{dc},
$$\mathbb{R}_{p,q}=\mathbb{D}_{+}^p\widehat{\otimes}_{\mathbb{R}}\mathbb{D}_{-}^q
\,{\stackrel{\sigma}\approx}\left\{\aligned\mathbb{D}^{p-q}_+,&\, \text{\quad\quad if } p\ge q;\\
\mathbb{D}_-^{q-p},&\,\text{\quad\quad otherwise}.\endaligned\right.
$$
\indent Now by \cite[\S 5.6]{porteous95}, $\mathbb{H}\widehat{\otimes}_{\mathbb{R}}\mathbb{H}$ is isomorphic to the $4\times 4$ real matrix superalgebra concentrated
on degree zero, which is (graded) Morita equivalent to  $\mathbb{R}$. Consequently, by Lemma~\ref{dddd}, $\mathbb{D}_{\pm}^8$ is graded Morita equivalent to $\mathbb{R}$.  Therefore Lemmas~\ref{pskewtensor} and
\ref{dddd} imply that
\begin{enumerate}\item if $p\geq q$ then $\mathbb{D}^{p-q}_+\,\approx_{\sigma}\left\{\aligned\mathbb{D}_{+}^i,\quad&\,\text {\quad\quad if } p-q\equiv
i(\text{mod\,}8), i=0, 1, 2, 3,\\\mathbb{D}_{+}^i\widehat{\otimes}_{\mathbb{R}}\mathbb{H},&\,
\text {\quad\quad if } p-q\equiv 4+i(\text{mod\,}8), i=0, 1, 2, 3;
\endaligned\right. $
\vspace{1\jot}
\item if $p<q$ then $\mathbb{D}^{q-p}_-\,\approx_{\sigma}\left\{\aligned\mathbb{D}_{-}^i,\quad&\,\text {\quad\quad if } q-p\equiv
i(\text{mod\,}8), i=0, 1, 2, 3,\\\mathbb{D}_{-}^i\widehat{\otimes}_{\mathbb{R}}\mathbb{H},&\,
\text {\quad\quad if } q-p\equiv 4+i(\mathrm{mod\,}8), i=0, 1, 2, 3;
\endaligned\right. $\end{enumerate}
By the proof of Lemma~\ref{dddd}, $\mathbb{D}_{+}\widehat{\otimes}_{\mathbb{R}}\mathbb{H}\cong
\mathbb{D}_{-}^3$ and $\mathbb{D}_{-}\widehat{\otimes}_{\mathbb{R}}\mathbb{H}\cong
\mathbb{D}_{+}^3$. Applying Lemmas~\ref{pskewtensor} and \ref{dc} again, we have  $\mathbb{D}_{+}^{i}\widehat{\otimes}_{\mathbb{R}}\mathbb{H}\approx_{\sigma}\mathbb{D}_{-}^{4-i}$ and  $\mathbb{D}_{-}^{i}\widehat{\otimes}_{\mathbb{R}}\mathbb{H}\approx_{\sigma}\mathbb{D}_{+}^{4-i}$ for $1\le i\le 3$.
Noticing that $q-p\equiv 4+i (\mathrm{mod\,}8)$ if and only if $p-q\equiv 4-i (\mathrm{mod\,}8)$.
As a consequence we complete the proof of (i) by \cite[Proposition~1.6]{abs}.

Note that $\widehat{\mathbb{R}_{p,q,r}}\cong \mathbb{R}_{q,p,r}$. Hence (ii) is follows directly by Lemma~\ref{rsp} and (i). \end{proof}

\begin{remark}Note that the $\mathbb{Z}_2$-graded Hochschild and cyclic (co)homology $H_*(A)$ of a finite dimensional superalgebra $A$ is graded Morita equivalent invariant \cite{zhao}. Theorem~\ref{tclifford} implies that $$H_{*}(\mathbb{R}_{p,q})=\left\{\aligned
&0,&&\text{if }*\geq 1;\\
&\mathbb{R},&&\text{if } *=0,
\endaligned\right.$$
since $\mathbb{R}$, $\mathbb{H}$ and  $\mathbb{D}_{\pm}^i$ for all $1\leq i\leq 3$ are separable superalgebras (see
\cite[Chap.~IV-V]{bass}, which gives a slight generalization of \cite[Proposition 1]{kassel} since the quadratic form  $\rho(x)=\sum_{i=1}^px_i^2-\sum_{j=1}^qx_{p+j}^2$ on $\mathbb{R}^{p+q}$ is degenerate.
\end{remark}

\begin{point}{}* Let $p$ and $q$ be positive integers. We denote by
$\mathbb{C}^{p, q}$ the  complex quadratic space $\mathbb{C}^{p+q}$
with a quadratic form $x_1^2+\cdots+ x_{p}^2$. The
Clifford superalgebra $\mathbb{C}_{p, q}$  on $\mathbb{C}^{p, q}$ is the complex unitary
superalgebra generated by odd generators $e_1, \cdots, e_{p+q}$ subject to the following relations:
\begin{align*}& e_ie_j + e_je_i = 0 &&\text{for } 1\le i\neq  j\le p+q;\\
& e_i^2=1&&\text{for }1\leq i\leq p;\\
&e_{p+j}^2=0&&\text{for } 1\leq j\leq q.\end{align*}
Note that the Clifford superalgebra $\mathbb{C}_{0, q}$ is the complex
Grassman superalgebra $\bigwedge_{\mathbb{C}}(q)$ generated by odd generators $e_{p+1}$, $\dots$, $e_{p+q}$.
Observe that the orthogonal primitive idempotent of $\bigwedge_{\mathbb{C}}(q)$ is the unity 1,
which implies that $\bigwedge_{\mathbb{C}}(q)$ are graded basic superalgebras for all positive integer $q\ge 1$.

We denote by $D=\mathbb{C}\oplus \mathbb{C}\varepsilon$ the complex
superalgebra generated by odd generators $\varepsilon$ subject to the relation
$\varepsilon^2=1$. For positive integer $n$, let $D^n$ be the superalgebra
$D\widehat{\otimes}_{\mathbb{C}}\cdots\widehat{\otimes}_{\mathbb{C}}
D$ ($n$ factors). Then $D^p\cong \mathbb{C}_{p,0}$ and $\mathbb{C}_{p, q}\cong
D^p\widehat{\otimes}_{\mathbb{C}}\bigwedge_{\mathbb{C}}(q)$  for all positive integers $p$ and $q$, see  \cite[Proposition~1.6]{abs}.\end{point}

\begin{lemma}\label{DD}
$D\widehat{\otimes}_{\mathbb{C}}D$ is graded Morita equivalent to
the superalgebras $\mathbb{C}$.
\end{lemma}

\begin{proof}Set $A=D\widehat{\otimes}_{\mathbb{C}}D$
Let $\varepsilon_{\pm}=\frac{1}{2}(1\otimes 1\pm
\mathbbm{i}\varepsilon\otimes\varepsilon)$, where $\mathbbm{i}$ is the
imaginary unit of $\mathbb{C}$. Then $\varepsilon_{\pm}^2=\varepsilon_{\pm}$, $\varepsilon_{+}\varepsilon_{-}=0=\varepsilon_{-}\varepsilon_{+}$ and $\varepsilon_{+}+\varepsilon_{-}=1\otimes 1$, which implies that  $\{\varepsilon_{+}, \varepsilon_{-}\}$ is the
complete set of orthogonal primitive idempotents of $D^2$ since $\dim_{\mathbb{C}}A_{0}=2$. It is straightforward to show that \begin{align*}\varepsilon_{+}A\varepsilon_{-}=\{c(\mathbbm{i}\otimes\varepsilon-1\otimes \varepsilon)\mid c\in\mathbb{C}\}\quad\text{and}\quad
\varepsilon_{-}A\varepsilon_{+}=\{c(\mathbbm{i}\otimes \varepsilon+\varepsilon\otimes1)\mid c\in\mathbb{C}\}.
 \end{align*}
Let $x=\frac{1}{2}(\mathbbm{i}\otimes\varepsilon-1\otimes \varepsilon)$ and $y=\frac{1}{2}(\varepsilon\otimes1+\mathbbm{i}\otimes \varepsilon)$. Then  $x\varepsilon_{-}y=\varepsilon_{+}$ and $y\varepsilon_{+}x=\varepsilon_{-}$.
Applying Lemma~\ref{lequivalent},  $\varepsilon_+$ and $\varepsilon_-$ are
$\sigma$-equivalent orthogonal primitive idempotents of $A$. Thus Proposition-Definition~\ref{psbasicdef} implies
 that $A$ is graded Morita equivalent to $\varepsilon_{+}A\varepsilon_{+}=\mathbb{C}\varepsilon_{+}\cong\mathbb{C}$.
\end{proof}

Now we yield the graded Morita equivalent classification for all complex Clifford superalgebras.

\begin{theorem}\label{thm-complex}Assume that $p$ and $q$ are non-negative integers. Then $\mathbb{C}_{p,q}$ is graded Morita equivalent to graded basic superalgebra $\mathbb{C}_{p(\mathrm{mod\,}2),q}$.
\end{theorem}

\begin{proof}Note that $\mathbb{C}_{p, q}\cong
\mathbb{C}_{1, 0}\widehat{\otimes}_{\mathbb{C}}\bigwedge_{\mathbb{C}}(q)$ and $\mathbb{C}_{1, 0}\cong D$.
It follows immediately that $\mathbb{C}_{p,q}\approx_{\sigma}\mathbb{C}_{p(\mathrm{mod\,}2),q}$ by using Lemmas~\ref{DD} and \ref{pskewtensor}. On the other hand, note that $\widehat{\mathbb{C}_{p,q}}\cong \mathbb{C}_{p,q}$, we have $\mathbb{C}_{p,q}\approx_{\pi}\mathbb{C}_{p(\mathrm{mod\,}2),q}$.  We complete the proof.
 \end{proof}

\section{The Grothendieck groups of Clifford superalgebras}\label{application}
\begin{point}{}* Most of the important applications of Clifford Superalgebras come through a detailed understanding of their ($\mathbb{Z}_2$-graded) representations. This understanding follows rather easily form the graded Morita equivalent classification given by Theorems~\ref{tclifford} and \ref{thm-complex}. In this section, we shall restrict our attention to the Clifford superalgebras $\mathbb{R}_{p,q}$ and $\mathbb{C}_{p,0}$ for all positive integers $p$ and $q$ in order to simplify our presentation, which are exactly those Clifford (super)algebras play a fundamental role in \cite{abs,Lawson}.\end{point}

\begin{point}{}*We begin with some notations. For positive integers $p$ and $q$, let $\mathrm{Irr}_{p,q}$ be the set of
 all nonequivalent irreducible graded $\mathbb{R}$-module for $\mathbb{R}_{p,q}$ in the category $\Gr \mathbb{R}_{p,q}$.
  Similarly, let $\mathrm{Irr}_{p}^{\mathbb{C}}$ be the set of all nonequivalent irreducible graded $\mathbb{R}$-module
   for $\mathbb{C}_{p,0}$ in the category $\Gr \mathbb{C}_{p,0}$.

An object which will be of our interest is the following.  Let $K_{p,q}$ be the Grothedieck group of
the category of finite dimensional graded real representations of $\mathbb{R}_{p,q}$, and let
$K_{p}^{\mathbb{C}}$ be the Grothedieck group of the category of
finite dimensional graded complex representations of $\mathbb{C}_{p,0}$. Then $K_{p,q}$ and $K_{p}^{\mathbb{C}}$
are the free abelian groups generated by $\mathrm{Irr}_{p,q}$ and $\mathrm{Irr}_{p}^{\mathbb{C}}$ respectively.
\end{point}

From the classification of Theorems~\ref{tclifford} and \ref{thm-complex} we immediately conclude the following:
\begin{lemma}\label{lem-v_p,q}Let $v_{p,q}$ and $v_p^{\mathbb{C}}$ denote the cardinality of $\mathrm{Irr}_{p,q}$
and of $\mathrm{Irr}_{p}^{\mathbb{C}}$ respectively. Then $$v_{p,q}=\left\{\begin{array}{ll}
2, & \text{ if }p-q\equiv 0 (\mathrm{mod\,}4) \\
 1, & \hbox{otherwise}\end{array}\right. \quad \text{ and }\quad
v_{p}^{\mathbb{C}}=\left\{\begin{array}{ll}2, & \text{ if }p \text{ is odd}\\
                                              1, &\text{ if }p \text{ is even}.
                                            \end{array}
                                          \right.$$\end{lemma}
 Now we can obtain the main result of this section.
\begin{theorem}\label{thm: Grothendieck}Let $p$ and $q$ be positive integers.
\begin{enumerate}\item The elements $\mathrm{Irr}_{p,q}$, $v_{p,q}$ and $K_{p,q}$  are as given in the following table:

\vspace{1.5\jot}
\begin{tabular}{|c|c|c|c|c|c|c|c|c|}
  \hline
   $p-q(\mathrm{mod\,}8)$ & $0$ & $1$ & $2$ & $3$ & $4$ & $5$ & $6$&$7$ \\ \hline
  $\mathbb{R}_{p,q}/\!\!\approx$& $\mathbb{R}$ & $\mathbb{R}_{1,0}$ & $\mathbb{R}_{2,0}$ & $\mathbb{R}_{3,0}$ & $\mathbb{H}$ & $\mathbb{R}_{0,1}$ & $\mathbb{R}_{0,2}$ &$\mathbb{R}_{0,3}$\\
$\mathrm{Irr}_{p,q}/\!\!\approx$ & $\{\mathbb{R}, \sigma(\mathbb{R})\}$ & $\{\mathbb{R}_{1,0}\}$ & $\{\mathbb{R}_{2,0}\}$ & $\{\mathbb{R}_{3,0}\}$ &$\{\mathbb{H}, \sigma(\mathbb{H})\}$ & $\{\mathbb{R}_{0,1}\}$ & $\{\mathbb{R}_{0,2}\}$&$\{\mathbb{R}_{0,3}\}$ \\
  $v_{p,q}$ & $2$ & $1$ & $1$ & $1$ & $2$ & $1$ & $1$&$1$ \\
$K_{p,q}$ & $\mathbb{Z}\oplus\mathbb{Z}$ & $\mathbb{Z}$ & $\mathbb{Z}$ & $\mathbb{Z}$ &$\mathbb{Z}\oplus \mathbb{Z}$ & $\mathbb{Z}$ & $\mathbb{Z}$&$\mathbb{Z}$ \\
  \hline
\end{tabular}

\vspace{1\jot}
\item $v_{p}^{\mathbb{C}}$, $\mathrm{Irr}_{p}^{\mathbb{C}}$ and $K_{p}^{\mathbb{C}}$ are given by following table:

\vspace{1.5\jot}\begin{tabular}{|c|c|c|c|c|}
                  \hline
                                   $p(\mathrm{mod\,}2)$ & $\mathbb{C}_p/\!\!\approx$& $\mathrm{Irr}_{p}^{\mathbb{C}}/\!\!\approx$&$v_{p}^{\mathbb{C}}$&$K_{p}^{\mathbb{C}}$ \\ \hline
                  0 & $\mathbb{C}$ & $\{\mathbb{C}, \sigma(\mathbb{C})\}$&$2$&$\mathbb{Z}\oplus\mathbb{Z}$ \\
                  1 & $\mathbb{C}_{1,0}$ & $\{\mathbb{C}_{1,0}\}$&$1$&$\mathbb{Z}$ \\
                  \hline
                \end{tabular}
\end{enumerate}
\end{theorem}

\begin{proof}Note that the graded modules of a gr-divisional superalgebra without non-trivial odd part are some copies of this gr-divisional superalgebra. Then the theorem is a direct consequence of Theorems~\ref{tclifford} and  \ref{thm-complex} and Lemma~\ref{lem-v_p,q}.\end{proof}

We remark in closing that Atiyah, Bott and Shapiro obtained Theorem~\ref{thm: Grothendieck}
 by using the periodicity theorem for Clifford superalgebras \cite[\S~5]{abs}, and it is
 important to understand the Atiyah-Bott-Shapiro isomorphisms \cite[Theorem~11.5]{abs}, see \cite{abs} and \cite{Lawson} for more details.

\section*{Acknowledgements}

The note develops partly form a part of \cite{zhaophd}.
The author is very grateful to Professors Yang Han and Yingbo Zhang for their invaluable help. 
The author would like to thank the referees for their comments and suggestions, which contributed to improvements in the presentation of the results.

\bibliographystyle{amsplain}

\end{document}